\documentclass[a4paper]{amsart}
\usepackage{psfrag,graphicx,
}
\usepackage{amsfonts}
\usepackage{amsmath,amscd,amsthm,amssymb,amsxtra,latexsym,epsf, 
epic,eepic,mathrsfs,xcolor}
\usepackage{epsf,latexsym}
\usepackage{enumerate}
\usepackage{todonotes}
\usepackage[all]{xy}
\usepackage[pagebackref]{hyperref}

\newfont{\Bb}{msbm10} 
\newfont{\Bc}{msbm8}
\newcommand{\baseRing}[1]{\ensuremath{\mathbb{#1}}}
\newcommand{\FF}{\baseRing{F}}
\newcommand{\ZZ}{\baseRing{Z}}
\newcommand{\CC}{\baseRing{C}}
\newcommand{\NN}{\baseRing{N}}
\newcommand{\RR}{\baseRing{R}}
\newcommand{\QQ}{\baseRing{Q}}
\newcommand{\PP}{\baseRing{P}}

\newcommand\s{\mathscr}
\begin{document}
\def\vsn{\vskip 10pt\noindent}
\def\vvs{\vskip 10pt}
\def\alt{\mathrm{Alt}}
\def\ii{\'\i }
\def\x{\mathbf{x}}
\def\sign{\text{sign}}
\def\e{\epsilon}
\def\cc{\circ}
\def\ssm{\smallsetminus}
\def\ni{\noindent}
\def\ben{\begin{enumerate}}
\def\een{\end{enumerate}}
\def\bit{\begin{itemize}}
\def\eit{\end{itemize}}
\def\y{\item}
\def\nl{\newline}
\def\vs{\vskip 10pt}
\def\hs{\hskip}
\def\beq{\begin{equation}}
\def\eeq{\end{equation}}
\def\bc{\begin{center}}
\def\ec{\end{center}}
\def\ld{{\ldots}}
\def\cd{{\cdots}}


\def\d{\mathbf{d}}
\def\Gl{\mbox{Gl}}
\def\L{\mathfrak{l}}
\def\pgl{\mathfrak{pgl}}
\def\gl{\mathfrak{gl}}
\def\gg{\mathfrak{g}}
\def\la{\langle}
\def\ra{\rangle}
\def\mm{\mathfrak m}


\def\th{{\theta}}
\def\ve{{\varepsilon}}
\def\pp{{\varphi}}
\def\a{{\alpha}}
\def\l{\left}
\def\r{\right}
\def\rr{{\rho}}
\def\ss{{\sigma}}
\def\g{{\gamma}}
\def\G{{\Gamma}}
\def\o{{\omega}}
\def\O{{\Omega}}
\def\k{{\kappa}}


\def\Proof{{\bf Proof}\hspace{.2in}}
\def\eop{\hspace*{\fill}$\Box$ \vskip \baselineskip}
\def\eoq{{\hfill{$\Box$}}}

\def\bu{\bullet}
\def\tr{{\pitchfork}}
\def\b{{\bullet}}
\def\w{{\wedge}}
\def\p{{\partial}}
\def\coker{\text{Coker}}

\def\supp{\mbox{supp}}
\def\im{\mbox{im}}
\def\Der{\mbox{Der}}

\def\Mat{\mbox{Mat}}
\def\Rep{\mbox{Rep}}
\def\sgn{\mbox{sign}}
\def\cod{\mbox{codim }\ }
\def\hh{\mbox{height }\ }
\def\opp{\mbox{\scriptsize opp}}
\def\vol{\mbox{vol}}
\def\deg{\mbox{deg}}
\def\Sym{\mbox{Sym}}
\def\Ext{\mbox{Ext}}
\def\Hom{\mbox{Hom}}
\def\End{\mbox{End}}
\def\Aut{\mbox{Aut}}
\def\Dh{{\Der(\log h)}}
\def\pd{\mbox{pd}}
\def\dim{{\mbox{dim}}}
\def\depth{{\mbox{depth}}}

\def\lb{\left(}
\def\rb{\right)}
\def\ji{{j_{i}}}
\def\jii{{j_{i'}}}
\def\EE{{\cal E}}
\def\R{{\cal R}}
\def\L{{\cal L}}
\def\C{{\cal C}}
\def\A{{\cal A}}
\def\TT{{\cal T}}
\newcommand{\bbbz}{{\mathbb Z}}


\newcommand{\n}[1]{\| #1 \|}
\newcommand{\um}[1]{{\underline{#1}}}
\newcommand{\om}[1]{{\overline{#1}}}
\newcommand{\fl}[1]{\lfloor #1 \rfloor}
\newcommand{\ce}[1]{\lceil #1 \rceil}
\newcommand{\ncm}[2]
{{\left(\!\!\!\begin{array}{c}#1\\#2\end{array}\!\!\!\right)}}
\newcommand{\ncmf}[2]
{{\left[\!\!\!\begin{array}{c}#1\\#2\end{array}\!\!\!\right]}}
\newcommand{\ncms}[2]
{{\left\{\!\!\!\begin{array}{c}#1\\#2\end{array}\!\!\!\right\}}}


\def\iff{{\Leftrightarrow}}
\def\imp{{\Rightarrow}}
\def\to{{\ \rightarrow\ }}
\def\too{{\ \longrightarrow\ }}
\def\into{{\hookrightarrow}}
\def\D{\Delta}
\def\st{\text{star}}


\newtheorem{theorem}{Theorem}[section]
\newtheorem{lemma}[theorem]{Lemma}
\newtheorem{sit}[theorem]{}
\newtheorem{lemmadefinition}[theorem]{Lemma and Definition}
\newtheorem{proposition}[theorem]{Proposition}
\newtheorem{corollary}[theorem]{Corollary}

\newtheorem{example}[theorem]{Example}
\newtheorem{question}[theorem]{Question}
\theoremstyle{definition}
\newtheorem{definition}[theorem]{Definition}
\newtheorem{conjecture}[theorem]{Conjecture}
\newtheorem{remark}[theorem]{Remark}
\title{Multiple points of a simplicial map and image-computing spectral sequences}
\author{Jos\'e-Luis Cisneros Molina}
\address{Instituto de Matem\'aticas, Unidad Cuernavaca\\ Universidad Nacional Aut\'onoma de M\'exico\\ Av. Universidad s/n, Col. Lomas de Chamilpa\\ Cuernavaca, Morelos, Mexico.}
\email{jlcisneros@im.unam.mx}
\thanks{The first author is Regular Associate of the Abdus Salam International Centre for Theoretical Physics, Trieste, Italy, supported by project CONACYT~253506.}
\author{David Mond}
\address{Mathematics Institute\\ University of Warwick\\ Coventry CV4 7AL, U.K.}  
\email{d.m.q.mond@warwick.ac.uk}

\date{\today}
\begin{abstract}
The Image-Computing Spectral Sequence computes the homology of the image of a finite map
from the alternating homology of the multiple point spaces of the map. 
A related spectral sequence was obtained by Gabrielov, Vorobjob and Zell in \cite{GVZ:BNSASPS} which computes the homology of the image of a closed map
from the homology of $k$-fold fibred products of the map.
We give new proofs of these results, in case the map can be triangulated.  Thanks to work of
Hardt in \cite{hardt}, this holds for a very wide range of maps, and in particular for most of the finite maps of interest in singularity theory. The proof seems conceptually simpler and
more canonical than earlier proofs.
\end{abstract}
\maketitle

\section{Introduction}

Spectral sequences are a powerful tool to compute (co)homology groups. One way to obtain spectral sequences is from a doble complex: it has associated a total complex which has two canonical
filtrations, given respectively, by the columns and rows of the double complex; each of such filtrations gives rise to a spectral sequence, both of them converging to the homology of the total complex.
One particular case which is very useful, is when one (or both) of the spectral sequences \textit{collapses onto one of the axis}, that is, all the rows (or colums) of the first page of the spectral sequence
are exact. In this case it is trivial to see what is the homology of the total complex. 
Classical examples of this situation are the \v Cech-de Rham (double) complex, which is used to prove that \v Cech cohomology
is isomorphic to de Rham cohomology \cite[Example~14.16]{Bott-Tu:DFAT}, and the Hochschild-Serre spectral sequence which relates the homology of a group to the homology of a normal subgroup and the 
corresponding quotient group \cite[\S VII.6]{Brown:CohoGr}.

The Image-Computing Spectral Sequence (ICSS), computes the homology of the image of a finite map $f\colon X\to Y$
from the alternating homology of the multiple point spaces $D^k(f)$ of $f$. It was introduced by Victor Goryunov and the second author in \cite{gm}, for rational cohomology, and
further developed by Goryunov, in \cite{gor}, for integer homology, to study the topology of the image of a stable perturbation of a map germ $\CC^n\to\CC^p$ with $n<p$.
Besides Singularity Theory \cite{houstontop,houstonspecseq,Hous07,Houston:TACS}, the ICSS has been used in Convex Geometry \cite{Kalai-Meshulam,Colin-etal:Helly}. 
In \cite{GVZ:BNSASPS,Basu-Riener:Equiv-Betti}, a related spectral sequence, which we call the GVZ Spectral Sequence (GVZSS), was used to give bounds of Betti numbers of semialgebraic sets.
It computes the image of a closed  map $f\colon X\to Y$ from the homology of $k$-fold fibred products $W^k(f)$ of $X$. 

The aim of this article is to obtain from a finite map $f\colon X\to Y$, the GVZSS and the ICSS, respectively, from the double complex
of chains of $k$-fold fibred products of $X$ and the double complex of alternating chains of multiple point spaces of $f$, proving that their first spectral sequences
collapse onto the $p$-axis. 
These proofs seem conceptually simpler and more canonical than earlier proofs.

As usual, computing anything directly with singular chains is essentially
impossible, and so for the ICSS, Goryunov in \cite{gor} worked instead with alternating cellular homology. In this paper, the new approach is to work
with (alternating) \emph{simplicial} homology, assuming a triangulation  of the map $f$ -- that is,  simplicial complexes $K$ and $L$ and a simplicial map $F:|K|\to |L|$, 
together with homeomorphisms $|K|\simeq X, |L|\simeq Y$ giving a commutative diagram 
$$\xymatrix{|K|\ar[d]_F\ar[r]^\simeq&X\ar[d]^f\\
|L|\ar[r]^\simeq&Y}.$$ 
We show in Section \ref{fsm} that a triangulation of $f$ gives rise to  triangulations of $W^k(f)$ and $D^k(f)$ for each $k$, and that these 
triangulations fit together in a remarkably simple way. We will refer to them as the triangulations {\it associated to the simplicial structure of $f$}.

\section{Double complexes of multiple point spaces}\label{alt.hom}

Let $f\colon X\to Y$ be a continuous surjective map with finite fibres. For each $k\geq 0$, the {\it $k$-fold product of $X$ fibred over $f$} is the subspace of $X^k$ defined by
\begin{equation*}
W^k(f)=\{(x_1,\dots,x_k)\in X^k\,|\, f(x_1)=\dots=f(x_k)\}.
\end{equation*}
The symmetric group $S_k$ acts on $W^k(f)$ and $D^k(f)$ permuting the copies of $X$.

Let $\pi_\ell\colon W^k(f)\to X$ with $\ell=1,\dots,k$ be the restriction of the $\ell$'th standard projection. Consider the map $f^{(k)}\colon W^k(f)\to Y$ defined by 
$f^{(k)}(\mathbf{x})=f\bigl(\pi_\ell(\mathbf{x})\bigr)$ for some $\ell$ with $1\leq\ell\leq k$. Evidently this is independent of the choice of $\ell$, so for convenience we take $\ell=1$.
There are projections $\ve^{i,k}\colon W^k(f)\to W^{k-1}(f)$, $i=1,\ld, k$, forgetting the $i$'th component.

Let $C_n(W^k(f))$ be the usual free abelian group of singular $n$-chains in $W^k(f)$ and consider the singular chain complex $(C_\bullet(W^k(f),\partial^k_\bullet)$ with the usual boundary map
\begin{equation*}
\p_n^k\colon C_n(W^k(f))\to C_{n-1}(W^k(f)).
\end{equation*}

The map $\ve^{i,k}\colon Z^k\to Z^{k-1}$ induces a morphism of chain complexes, 
\begin{equation*}
\ve^{i,k}_{\#}\colon C_\bu(W^k(f))\to C_\bu(W^{k-1}(f)).
\end{equation*}
We denote by $\ve^{i,k}_{\#,n}\colon C_n(Z^k)\to C_n(Z^{k-1})$ the corresponding homomorphism on $n$-chains. 
Define the morphism $\rho^k_n\colon C_n(W^k(f))\to C_n(W^{k-1}(f))$ by
\begin{equation}\label{eq:rho}
\rho^k_n=\sum_{i=1}^k(-1)^{i-1}\ve^{i,k}_{\#,n}.
\end{equation}
In Section~\ref{fsm}  we will prove that $\rho^{k-1}_n\circ\rho^k_n=0$ and $f_{\#,n}\circ\rho^2_n=0$ (Lemmata~\ref{lem:rho.rho} and \ref{lem:f.rho}), hence for each $n$, the sequence of $n$-chains
\begin{equation}\label{eq:wcc}
\dots\to C_n(W^k(f))\xrightarrow{\rho^k_n}\cdots
\xrightarrow{\rho^3_n}C_n(W^2(f))\xrightarrow{\rho^2_n}C_n(X)\xrightarrow{f_{\#,n}}C_n(Y)\to0,
\end{equation}
is a chain complex.

\subsection{Alternating homology of multiple point spaces}

We also define the $k$'th {\it multiple point space} $D^k(f)$ as the closure, in $X^k$, of the set of strict multiple points\footnote{A slight variant on this definition is
used in e.g. \cite{GaffArc}, \cite{nunpen} when one wishes to describe how the multiple points change when one deforms a map. In 
this paper we will not need to do so, and so we use the simpler definition.}
\begin{equation*}
D_{\text{S}}^k(f)=\{(x_1,\ld,x_k)\in X^k: x_i\neq x_j\ \text{if }i\neq j, f(x_i)=f(x_j)\ \text{for all }i,j\}.
\end{equation*}
Notice that $D^k(f)$ is a subspace of $W^k(f)$ and the action of $S_k$ restricts to $D^k(f)$.

We also consider the restrictions to $D^k(f)$ of the maps $\pi_\ell$, $f^{(k)}$ and $\ve^{i,k}$, $i=1,\ld, k$.
Since all the maps $\ve^{i,k}$ are left-right equivalent to one another thanks to the symmetric group actions, when restricted to $D^k(f)$,
we will use only $\ve^{k,k}$, which we abbreviate to $\ve^k$. 
Thus, we obtain a {\it tower of multiple point spaces}
\begin{equation}\label{eq:tmp}
\dots\to D^{k+1}(f)\xrightarrow{\varepsilon^{k+1}}D^k(f)\xrightarrow{\varepsilon^{k}}\cdots\xrightarrow{\varepsilon^{3}}D^2(f)\xrightarrow{\varepsilon^{2}}X\xrightarrow{f}Y.
\end{equation}

Let $Z^k$ denote $W^k(f)$ or $D^k(f)$. 
Let $C_n(Z^k)$ be the usual free abelian group of singular $n$-chains in $Z^k$. We define the group of {\it alternating $n$-chains} $C_n^\alt(Z^k)$ by
\begin{equation*}
C_n^\alt(Z^k)=\{c\in C_n(Z^k): \sigma_\#(c)=\sign(\sigma)c\ \text{for all }\sigma\in S_k\},
\end{equation*}
and the {\it alternating chain complex} as $C^\alt_\bu(Z^k)$ with the usual boundary map
\begin{equation*}
\p_n^k\colon C_n^\alt(Z^k)\to C_{n-1}^\alt(Z^k).
\end{equation*}
It is a subcomplex of the usual singular chain complex. We call its homology the {\it alternating homology of $Z^k$} and denote it by $A H_*(Z^k)$. 
The notation $\alt H_*(Z^k)$ is used in \cite{gor}; we prefer ours, since $\alt H_*(Z^k)$ can be misunderstood as the image of an operator $\alt:H_*(Z^k)\to H_*(Z^k)$. 

\begin{remark}
Kevin Houston proved in \cite[Theorem~3.4]{houstonspecseq} that the alternating homologies of $W^k(f)$ and $D^k(f)$ are canonically isomorphic:
\begin{equation*}
AH_n(W^k(f))=AH_n(D^k(f)),
\end{equation*}
since on the diagonals, i.~e., on $W^k(f)\setminus D^k(f)$ there are no alternating chains \cite[Theorem~2.7]{houstonspecseq}.
\end{remark}

Let $d_j\colon\{1,\dots,k-1\}\to\{1,\dots,k\}$ be the monotone non-decreasing inclusion which does not have $j$ in its image, i.~e., given by
\begin{equation*}
 d_j(i)=\begin{cases}
         i & \text{if $i<j$,}\\
         i+1 & \text{if $i\geq j$.}
        \end{cases}
\end{equation*}
Given a permutation $\sigma\in S_{k-1}$ we define the permutation $\bar{\sigma}^j\in S_k$ given by
\begin{equation*}
\bar{\sigma}^j(i)=\begin{cases}
		  d_j(\sigma(i)) & \text{if $i<j$,}\\
		  j & \text{if $i=j$,}\\
		  d_j(\sigma(i-1)) & \text{if $i>j$.}
		  \end{cases}
\end{equation*}
We have that $\sign(\bar{\sigma}^j)=\sign(\sigma)$.
By a straightforward computation we get
\begin{equation}\label{eq:bar.sigma}
\sigma\circ \ve^{j,k} =\ve^{j,k}\circ \bar{\sigma}^j. 
\end{equation}

The chain morphisms $\rho^k_n$ and $\ve^{i,k}_{\#}$ restrict to a chain morphism of the subcomplex of alternating chains:
\begin{lemma}\label{lem:r.s.alt}
There are the following inclusions:
\begin{enumerate}[1.]
 \item $\rho^k_n(C^\alt_n(W^k(f))\subset C^\alt_n(W^{k-1}(f))$,
 \item $\ve^{j,k}_{\#,n}(C_n^\alt(D^k(f))\subset C_n^\alt(D^{k-1}(f))$ for $j=1,\dots,k$.\label{it:eps.alt}
\end{enumerate}
\end{lemma}

\begin{proof} 
\begin{enumerate}[1.]
 \item By \eqref{eq:bar.sigma}, for any $n$-chain $c$ in $W^k(f)$ and any $\sigma\in S_{k-1}$ we have
\begin{equation}\label{eq:sigma.epsilon}
\sigma_\#(\ve^{j,k}_{\#,n}(c))=\ve^{j,k}_{\#,n}(\bar{\sigma}^j_\#(c)),\quad\text{for $j=1,\dots,k$.}
\end{equation}
From \eqref{eq:sigma.epsilon} and the definition of $\rho^k_n$ given in \eqref{eq:rho} it follows that
\begin{equation*}
\sigma_\#(\rho^k_n(c))=\rho^k_n(\bar{\sigma}^j_\#(c)).
\end{equation*}
If $c$ is alternating
\begin{equation*}
\sigma_\#(\rho^k_n(c))=\rho^k_n(\bar{\sigma}^j_\#(c))=\rho^k_n(\sign(\bar{\sigma}^j)(c))=\sign(\sigma)\rho^k_n(c).
\end{equation*}
\item Taking in \eqref{eq:sigma.epsilon} $c$ an alternating $n$-chain in $D^k(f)$ we have
\begin{equation*}
\sigma_\#(\ve^{j,k}_{\#,n}(c))=\ve^{j,k}_{\#,n}(\bar{\sigma}^j_\#(c))=\ve^{j,k}_{\#,n}(\sign(\bar{\sigma}^j)(c))=\sign(\sigma)\ve^{j,k}_{\#,n}(c).
\end{equation*}
\end{enumerate}
\end{proof}
It is convenient to include $X$ and $Y$ in the tower of multiple points spaces, as $D^1(f)$ and $D^{0}(f)$, and in this context to regard $f:X\to Y$ as $\ve^1:D^1(f)\to D^0(f)$. 
\begin{lemma}\label{isacomplex}
$\ve^{k-1}_\#\circ\ve^k_\#=0$ on $C_\bu^\alt(D^k(f))$.
\end{lemma}

\begin{proof}
Consider the action of the odd permutation $(k-1,k)$. We have 
$$\ve^{k-1}\circ\ve^k\circ (k-1,k)=\ve^{k-1}\circ\ve^k,$$
and so 
$$\ve^{k-1}_{\#,n}\circ\ve^k_{\#,n}\circ (k-1,k)_\#=\ve^{k-1}_{\#,n}\circ\ve^k_{\#,n}.$$
But if $c$ is an alternating $n$-chain, 
$$\ve^{k-1}_{\#,n}\circ\ve^k_{\#,n}\circ (k-1,k)_\#(c)=\ve^{k-1}_{\#,n}\circ\ve^k_{\#,n}(-c)=-\ve^{k-1}_{\#,n}\ve^k_{\#,n}(c).$$
\end{proof}
It follows from Lemma~\ref{lem:r.s.alt}--\ref{it:eps.alt} with $j=k$ and Lemma~\ref{isacomplex} that for each $n$, 
the sequence induced on alternating $n$-chains by the tower of multiple point spaces \eqref{eq:tmp},
\begin{equation}\label{eq:rcy}
\dots\to C_n^\alt(D^k(f))\xrightarrow{\ve^{k}_{\#,n}}\cdots
\xrightarrow{\ve^{3}_{\#,n}}C_n^\alt(D^2(f))\xrightarrow{\ve^{2}_{\#,n}}C_n(X)\xrightarrow{f_{\#,n}}C_n(Y)\to0,
\end{equation}
is a chain complex.

In Section~\ref{DC-ICSS}, using a standard sign trick, we will obtain two first quadrant double complexes 
$(C_n(W^k(f)),\p_n^k, \varrho^k_n)$ and $(C^\alt_n(D^k(f)),\p_n^k, \epsilon^k_{\#,n})$
from the commutative diagrams of free abelian groups and homomorphisms given by \eqref{eq:wcc} and \eqref{eq:rcy}.
Given any  first quadrant double complex,  two standard spectral sequences compute the homology of the total complex. 
The GVZSS and ICSS are the spectral sequences resulting from these double complexes by first applying the differential $\p_q$ to get, respectively, 
$E^1_{p,q}=H_q(W^{p+1}(f))$ and $E^1_{p,q}=AH_q(D^{p+1}(f))$ with differentials $(\varrho^{p+1}_q)_*$ and $(\epsilon^{p+1}_q)_*$. 
The opposite spectral sequences, in which one first applies the differentials $\varrho^q$ and  $\epsilon^q$, are hard to make sense of with (alternating) singular homology. 
Our main result here is the following.
\begin{theorem}\label{thm:mt}
Let $f\colon X\to Y$ be a finite surjective simplicial map. 
Let $Z^k$ be $W^k(f)$ or $D^k(f)$.
Then each projection $\ve^{i,k}\colon Z^k\to Z^{k-1}$ is a simplicial map with respect to the associated triangulations of the $Z^k$, and for each $n\in \NN$, the 
resulting complexes \eqref{eq:wcc} and  \eqref{eq:rcy}, now with (alternating) \emph{simplicial} chains, are acyclic resolutions of $C_n(Y)$.
\end{theorem}
From this it follows immediately that the opposite spectral sequences collapse onto the $p$-axis and have $E^1_{p,0}=C_q(Y)$ with differential the usual boundary operator, thus showing that 
the homology of the total complex is that of $Y$.  \vsn
\ni {\it Here, and in what follows,  $X$ and $Y$ are geometric \emph{simplicial} complexes, and $C_n$ and $C_n^\alt$ denote groups of \emph{simplicial} chains, and alternating \emph{simplicial} chains.}
\vskip 10pt
Before proceeding with the proof, let us point out that the images of ``generic'' smooth or complex analytic maps $f\colon M^n\to N^p$
(where $M$ and $N$ are smooth or complex analytic manifolds with $n<p$) are in general very singular, whereas it turns
out that the multiple point spaces $D^k(f)$ are in general less so, and this makes their homology more accessible than that of the image itself. If, for example, $f$ is right-left stable 
in the sense of Thom and Mather, e.g. \cite{mather2}, and at all non-immersive points, $\dim\ker d_xf\leq 1$,  then the $D^k(f)$ are smooth. This condition on kernel rank holds for all right-left stable maps if $n<6$.  We say more about the applications of the ICSS in Singularity Theory in Section \ref{singularitytheory} below.

\section{The GVZSS and the ICSS from the double complexes}\label{DC-ICSS}

In this section, given a finite surjective simplicial map $f\colon X\to Y$ between geometric simplicial complexes $X$ and $Y$, using Theorem~\ref{thm:mt} we obtain the GVZSS and  the ICSS from, 
respectively, the second spectral sequence associated to the double complexes given by the complexes \eqref{eq:wcc} and \eqref{eq:rcy}.

\subsection{Spectral sequences arising from a double complex}

Here we summarise the definitions and main results on spectral sequences that we use. Good references for spectral sequences are \cite{McCleary,Rotman,Weibel:HomAlg}. For some intuition on the ICSS which can help navigate this material, see \cite[\S 1.2 and \S1.6]{mond15}.

A {\it homology spectral sequence} consists of the following data:
\begin{enumerate}
 \item A family $\{E^r_{p,q}\}$ of modules for all integers $p$, $q$ and $r\geq1$. The {\it total degree} of the term $E^r_{p,q}$ is $n=p+q$.
 \item Homomorphisms $d^r_{p,q}\colon E^r_{p,q}\to E^r_{p-r,q+r-1}$ such that $d^r\circ d^r=0$.
 \item Isomorphisms between $E^{r+1}_{p,q}$ and the homology of $E^r_{*,*}$ at the term $E^r_{p,q}$:
$$E^{r+1}_{p,q}\cong\frac{\ker d^r_{p,q}}{\im\ d_{p+r,q-r+1}}.$$
\end{enumerate}
The collection of modules $E^r_{p,q}$ for fixed $r$, together with the morphisms $d^r_{p,q}$, is often referred to as the $r$'th {\it page} of the spectral sequence. Thus, item 3 above says that the
$r+1$'st page is obtained by taking the homology of the $r$'th page.

A {\it first quadrant spectral sequence} is one with $E^r_{p,q}=0$ unless $p\geq0$ and $q\geq0$. If this condition holds for $r=1$, then it holds for all $r\geq1$. For fixed $p$ and $q$, $E^r_{p,q}=E^{r+1}_{p,q}$ for all $r>\max\{p,q+1\}$, because then every arrow $d^r_{p,q}$ beginning at $E^r_{p,q}$ must end at $0$, and every arrow ending at $E^r_{p,q}$ must begin at 0. We write $E^\infty_{p,q}$ for this stable value of $E^r_{p,q}$.

A homology spectral sequence is {\it bounded} if for each $n$ there are only finitely many nonzero terms of total degree $n$ in $E^1_{*,*}$. Thus a first quadrant spectral sequence is bounded. 
A bounded spectral sequence {\it converges} to a graded module $H_*$, denoted by $E^r_{p,q}\Longrightarrow H_{p+q}$, if for each $n$, $H_n$ has a finite filtration
$$0=F^sH_n\subseteq\cdots\subseteq F^{p-1}H_n\subseteq F^pH_n\subseteq F^{p+1}H_n\subseteq\cdots\subseteq F^tH_n=H_n,$$
and there are isomorphisms
$$E^\infty_{p,q}\cong F^pH_{p+q}/F^{p-1}H_{p+q}.$$

A {\it double complex} is an ordered triple $(C,d',d'')$ where $C=(C_{p,q})$ is a bigraded module, $d'_{p,q}\colon C_{p,q}\to C_{p-1,q}$ and $d''_{p,q}\colon C_{p,q}\to C_{p,q-1}$ such that $d'_{p-1,q}\circ d'_{p,q}=0$, $d''_{p,q-1}\circ d''_{p,q}=0$ and
$$d'_{p,q-1}\circ d''_{p,q}+d''_{p-1,q}\circ d'_{p,q}=0.$$
A double complex $(C_{p,q})$ is a {\it first quadrant double complex} if $C_{p,q}=0$ whenever $p<0$ or $q<0$.

If $C$ is a double complex, its {\it total complex}, denoted by $\mathrm{Tot}(C)$, is the chain complex \cite[Lemma~10.5]{Rotman} with $n$'th term
$$\mathrm{Tot}(C)_n=\bigoplus_{p+q=n}C_{p,q},$$
with differentials $D_n\colon\mathrm{Tot}(C)_n\to\mathrm{Tot}(C)_{n-1}$ given by
$$D_n=\sum_{p+q=n} (d'_{p,q}+d''_{p,q}).$$
We can define two filtrations of $\mathrm{Tot}(C)$ by
$${}^IF^p(\mathrm{Tot}(C)_n)=\bigoplus_{i\leq p}C_{i,n-i},\qquad {}^{II}F^p(\mathrm{Tot}(C)_n)=\bigoplus_{j\leq p}C_{n-j,j}.$$
These two filtrations determine two spectral sequences \cite[Theorem~10.16]{Rotman}. If $C$ is a first quadrant double complex, both filtrations are bounded and both
spectral sequences converge to the homology of the total complex of $C$. The first filtration gives the \emph{first spectral sequence}
\begin{equation}\label{eq:I.SS}
\begin{aligned}
 {}^IE^0_{pq}&=C_{p,q},\quad \text{with differentials $d^0_{p,q}=d''_{p,q}$,}\\
 {}^IE^1_{pq}&=H''_q(C_{p,*}),\quad  \text{with differentials $d^1_{p,q}=(d'_{p,q})_*$,}\\[5pt]
 {}^IE^2_{pq}&=H'_pH''_q(C)\Rightarrow H_{p+q}(\mathrm{Tot}(C)).
 \end{aligned}
\end{equation}
The second filtration gives the \emph{second spectral sequence}
\begin{equation}\label{eq:II.SS}
\begin{aligned}
 {}^{II}E^0_{pq}&=C_{q,p},\quad \text{with  differentials $d^0_{p,q}=d'_{p,q}$,}\\
 {}^{II}E^1_{pq}&=H'_q(C_{*,p}),\quad  \text{with differentials $d^1_{p,q}=(d''_{p,q})_*$,}\\
 {}^{II}E^2_{pq}&=H''_pH'_q(C)\Rightarrow H_{p+q}(\mathrm{Tot}(C)).
\end{aligned}
\end{equation}
Here $(d'_{p,q})_*$ and $(d''_{p,q})_*$ denote respectively, the homomorphisms induced in homology by the differentials $d'_{p,q}$ and $d''_{p,q}$ of the double complex.
Both spectral sequences converge to the homology of the total complex of $C$. In both, the higher differentials involve both $d'$ and $d''$, and can be constructed by diagram chasing. 
In each of these cases, the filtration on the homology of the total complex is determined by the spectral sequence itself, and is not an additional independent datum.
If a first quadrant spectral sequence converges to $H_n$, then each $H_n$ has a finite filtration of length $n+1$
\begin{equation}\label{eq:filt.f.q}
0=F_{-1}H_n\subset F_0H_n\subset\dots\subset F_{n-1}H_n\subset F_nH_n=H_n.
\end{equation}
The bottom part $F_0H_n=E^\infty_{0n}$ lies on the $q$-axis, the intermediate quotients are $F_pH_n/F_{p-1}H_n\cong E^\infty_{p,n-p}$ and 
the top quotient $H_n/F_{n-1}H_n\cong E^\infty_{n0}$ lies on the $p$-axis.
    
A spectral sequence $(E^r,d^r)$ {\it collapses} onto the $p$-axis if $E^2_{p,q}=0$ for all $q\neq 0$. In this case $E^\infty_{p,q}=E^2_{p,q}$ 
and $H_n(\mathrm{Tot}(C))\cong E^2_{n,0}$ \cite[Proposition~10.21]{Rotman}. 
More generally, a spectral sequence {\it collapses at $E^r$} if all the differentials $d^s_{p,q}$ for $s\geq r$ and all $p,q$ are zero; 
in this case $E^s_{p,q}=E^r_{p,q}$ for $s\geq r,$ and so $E^\infty_{p,q}=E^r_{p,q}$.
\vskip 10pt

\subsection{The GVZSS and the ICSS}

Let $f\colon X\to Y$ be a finite surjective simplicial map. 
Since $\rho^k_n$ and $\ve^{i,k}_{\#,n}$ are chain morphism, the diagrams of free abelian groups and homomorphisms given by the arrays
\begin{equation*}
 (C_n(W^k(f)),\p_n^k, \rho^k_n),\qquad\text{and}\qquad (C^\alt_n(D^k(f)),\p_n^k, \ve^k_{\#,n})
\end{equation*}
commute, and every row and column is a chain complex. 
Applying the usual sign trick \cite[Example~10.4]{Rotman}, we define 
\begin{equation}\label{eq:sign-trick}
\varrho^k_n=(-1)^{n}\rho^k_n=\sum_{i=1}^k(-1)^{n+i-1}\ve^{i,k}_{\#,n},\qquad\text{and}\qquad \epsilon^k_n=(-1)^{n}\ve^k_{\#,n},
\end{equation}
to get (first quadrant) double complexes 
$(C_n(W^{k+1}(f)),\p_n^{k+1}, \varrho^{k+1}_n)$ with $n,k\geq0$:
\begin{equation}\label{eq:dc.w}
 \xymatrix{
\vdots \ar[d]_{\varrho^{k+1}_{0}}& \vdots\ar[d]_{\varrho^{k+1}_{1}} & & \vdots\ar[d]_{\varrho^{k+1}_{l}}\\
C_0(W^k(f))\ar[d]_{\varrho^k_{0}} & C_1(W^k(f))\ar[l]_>>>>>{\partial^k_1}\ar[d]_{\varrho^k_{1}} & \ar[l]_>>>>>{\partial^k_2}\cdots & C_l(W^k(f))\ar[l]_>>>>>{\partial^k_l}\ar[d]_{\varrho^k_{l}} & \ar[l]_<<<<{\partial^k_{l+1}}\cdots\\
\vdots \ar[d]_{\varrho^3_{0}}& \vdots\ar[d]_{\varrho^3_{1}} & & \vdots\ar[d]_{\varrho^3_{l}}\\
C_0(W^2(f))\ar[d]_{\varrho^2_{0}} & C_1(W^2(f))\ar[l]_<<<<{\partial^2_1}\ar[d]_{\varrho^2_{1}} & \ar[l]_>>>>>{\partial^2_2}\cdots & C_l(W^2(f))\ar[l]_>>>>>{\partial^2_l} \ar[d]_{\varrho^2_{l}}& \ar[l]_<<<<{\partial^2_{l+1}}\cdots\\
C_0(X) & C_1(X)\ar[l]_{\partial_1} & \ar[l]_<<<<{\partial_2}\cdots & C_l(X)\ar[l]_>>>>>{\partial_l} & \ar[l]_<<<<{\partial_{l+1}}\cdots\\
}
\end{equation}
and $(C^\alt_n(D^{k+1}(f)),\p_n^{k+1}, \epsilon^{k+1}_n)$ with $n,k\geq0$:
\begin{equation}\label{eq:dc}
 \xymatrix{
\vdots \ar[d]_{\epsilon^{k+1}_{0}}& \vdots\ar[d]_{\epsilon^{k+1}_{1}} & & \vdots\ar[d]_{\epsilon^{k+1}_{l}}\\
C_0^{\alt}(D^k(f))\ar[d]_{\epsilon^k_{0}} & C_1^{\alt}(D^k(f))\ar[l]_>>>>>{\partial^k_1}\ar[d]_{\epsilon^k_{1}} & \ar[l]_>>>>>{\partial^k_2}\cdots & C_l^{\alt}(D^k(f))\ar[l]_>>>>>{\partial^k_l}\ar[d]_{\epsilon^k_{l}} & \ar[l]_<<<<{\partial^k_{l+1}}\cdots\\
\vdots \ar[d]_{\epsilon^3_{0}}& \vdots\ar[d]_{\epsilon^3_{1}} & & \vdots\ar[d]_{\epsilon^3_{l}}\\
C_0^{\alt}(D^2(f))\ar[d]_{\epsilon^2_{0}} & C_1^{\alt}(D^2(f))\ar[l]_<<<<{\partial^2_1}\ar[d]_{\epsilon^2_{1}} & \ar[l]_>>>>>{\partial^2_2}\cdots & C_l^{\alt}(D^2(f))\ar[l]_>>>>>{\partial^2_l} \ar[d]_{\epsilon^2_{l}}& \ar[l]_<<<<{\partial^2_{l+1}}\cdots\\
C_0(X) & C_1(X)\ar[l]_{\partial_1} & \ar[l]_<<<<{\partial_2}\cdots & C_l(X)\ar[l]_>>>>>{\partial_l} & \ar[l]_<<<<{\partial_{l+1}}\cdots\\
}
\end{equation}

\begin{proposition}\label{prop:I.ss}
The first spectral sequences ${}^IE^r_{p.q}$ of the double complexes \eqref{eq:dc.w} and \eqref{eq:dc} collapse on the $p$-axis and $E^2_{n,0}\cong H_n(\mathrm{Tot}(C))\cong H_n(Y)$.
\end{proposition}

\begin{proof}
We give the proof for the double complex \eqref{eq:dc}, the proof for the double complex \eqref{eq:dc.w} is completely analogous.
By \eqref{eq:I.SS}, the first page of the first spectral sequence is given by taking the homology with respect to the vertical differentials $\epsilon^q$ in diagram \eqref{eq:dc}, i.e.,
\begin{equation*}
{}^IE^1_{p.q}\cong H_q((C^\alt_p(D^{\bu+1}(f)),\epsilon^\bu_p)).
\end{equation*}
By Theorem~\ref{thm:mt} the chain complex \eqref{eq:rcy} is exact, hence we have
\begin{equation*}
{}^IE^1_{p.q}=\begin{cases}
                     C_p(X)/\im\ \ve^{2}_{\#,p}\cong C_p(Y) & \text{if $q=0$,}\\
                     0      & \text{if $q\neq0$.}
                    \end{cases}
\end{equation*}
The differential $(\p_q)_*\colon {}^IE^1_{p.0}=C_p(X)/\im\ \ve^{2}_{\#,p}\to C_{p-1}(X)/\im\ \ve^{2}_{\#,p-1}={}^IE^1_{p-1.0}$ corresponds under the isomorphism to the 
boundary map $\p_q\colon C_p(Y)\to C_{p-1}(Y)$, thus we have
\begin{equation*}
{}^IE^2_{p.q}=\begin{cases}
                     H_p(Y) & \text{if $q=0$,}\\
                     0      & \text{if $q\neq0$.}
                    \end{cases}
\end{equation*}
Therefore, the spectral sequence ${}^IE^r_{p.q}$ collapses on the $p$-axis and $H_n(\mathrm{Tot}(C))\cong H_n(Y)$.
\end{proof}

\begin{corollary}
\begin{enumerate}[1.]
\item The second spectral sequence ${}^{II}E^1_{p.q}$ of the double complex \eqref{eq:dc} converges to $H_\bu(Y)$, i.e.,
\begin{equation}\label{eq:gvzss}
{}^{II}E^1_{p,q}\cong H_q(W^{p+1}(f))\Longrightarrow H_{p+q}(Y),
\end{equation}
 with differentials given by $(\varrho^{p+1}_q)_*$.
 \item The second spectral sequence ${}^{II}E^1_{p.q}$ of the double complex \eqref{eq:dc} converges to $H_\bu(Y)$, i.e.,
\begin{equation}\label{eq:icss}
{}^{II}E^1_{p,q}\cong AH_q(D^{p+1}(f))\Longrightarrow H_{p+q}(Y),
\end{equation}
with differentials given by $(\epsilon^{p+1}_q)_*$.
\end{enumerate}
\end{corollary}

\begin{proof}
 By \eqref{eq:II.SS}, the first page of the second spectral sequence is given by taking the homology with respect to the horizontal differentials $\p_q$ in diagram \eqref{eq:dc.w} (resp.\ in \eqref{eq:dc}),
so we get $H_q(W^{p+1}(f))$ (resp.\ $AH_q(D^{p+1}(f))$). It converges to the homology of the total complex, which by Proposition~\ref{prop:I.ss} is isomorphic to $H_\bu(Y)$.
\end{proof}

The second spectral sequence \eqref{eq:gvzss} of the double complex \eqref{eq:dc.w} is the \textit{GVZ Spectral Sequence}.
The second spectral sequence \eqref{eq:icss} of the double complex \eqref{eq:dc} is the {\it Image-Computing Spectral Sequence}.

\section{Finite simplicial maps}\label{fsm}

To complete the construction of the GVZSS and the ICSS it only remains to prove Theorem~\ref{thm:mt}. This is the aim of this section.

We suppose that $X$ and $Y$ are geometric simplicial complexes embedded in $\RR^N,\RR^P$ respectively, and that
$f:X\to Y$ is simplicial, surjective and finite-to-one. We begin by establishing some notation and proving some elementary lemmas. An {\it ordered simplex} is an affine homeomorphism
\begin{equation*}
(t_0,\ld,t_n)\mapsto \sum_it_iv_i
\end{equation*}
from the standard $n$-simplex into $X$ or $Y$ which forms part of its simplicial structure; the image of an ordered simplex
is a {\it geometric simplex}. Thus,  
an ordered simplex of $X$ is determined by a geometric $k$-simplex together with an ordering of its vertices. We will use the symbol $\Delta$ to denote an ordered simplex, 
and $|\Delta|$ to denote its image. When we need to speak of a geometric simplex without specifying an ordering of its vertices, we will use the symbol $\G$. 
We write $C_k(X)$ for the free abelian group generated by the ordered $k$-simplices of $X$. If $f$ is a simplicial map and $\D$ an ordered simplex of $X$, 
then $f(|\Delta|)$ is a geometric simplex of $Y$ and $f_\#(\D)$ an ordered simplex. Thus $f(|\Delta|)=|f_\#(\Delta)|$. If $\Delta=[v_0,\ld,v_n]$ then 
$|\mathring \D|$ denotes its {\it interior}:
\begin{equation*}
|\mathring\D|=\{\sum_it_iv_i:\sum_it_i=1, t_i>0, i=0,\ld, n\}.
\end{equation*}
The following properties of any simplicial complex $K$ are well known:
\begin{enumerate}
 \item Each point $x\in K$ lies in the interior of a unique geometric simplex of $K$, which we denote by $\G_x$. 
 \item We write $|\Delta|\leq|\Delta'|$ when $|\D|$ is a face of $|\D'|$, and note that
if $|\D'|\cap|\mathring \D|\neq 0$ then $|\D|\leq |\D'|$. For two simplices meet only along a common face of each, so
$\D=\overline{\mathring  \D}$ is a face of $\D'$.
\item For each point $x$,
\begin{equation*}
\st(x):=|\mathring\G_x|\cup\bigcup_{\G_x\leq|\Delta|}|\mathring\Delta|
\end{equation*}
is an open neighbourhood of $x$. For suppose that $(x_n)\to x$. By local finiteness of $K$ we can suppose that
all $x_n$ lie in the interior of one simplex $\Delta$. Then $x\in\Delta\cap\mathring \G_x,$ so by the preceding statement,
$\G_x\leq\D$, and thus all the $x_n$ lie in $\st(x)$.
\end{enumerate}

\begin{lemma}\label{basic1}
$f$ maps each geometric $n$-simplex of $X$ isomorphically to an $n$-simplex of $Y$.
\end{lemma}
\begin{proof} By definition of simplicial map,  on the simplex $\Delta:=[v_0,\ld,v_n]$, we have
\begin{equation*}
f(\sum_it_iv_i)=\sum_it_if(v_i).
\end{equation*}
If the vertices $f(v_0),\ld, f(v_n)$ were affinely dependent then this map would not be finite to one. Thus
$\Delta':=[f(v_0),\ld, f(v_n)]$ is an $n$-simplex of $Y$ and $f:|\Delta|\to |\Delta'|$ is an isomorphism.
\end{proof}

\begin{lemma}\label{cp} 
If $f(x_1)=f(x_2)$ then the barycentric coordinates of $x_1$ in $\G_{x_1}$, with respect to a suitable ordering of the vertices, are the same as those of $x_2$ in $\G_{x_2}$.
\end{lemma}
\begin{proof} Write $y:=f(x_1)=f(x_2)$. If $\G_y=|[w_1,\ld,w_n]|$ then for $i=1, 2$ it is possible to order the vertices $v^{(i)}_j$ of $\G_{x_i}$ so that $f(v^{(i)}_j)=w_j$. 
Since $f$ is linear on $\G_{x_i}$, the barycentric coordinates of 
$x_i$ in $\G_{x_i}$ are the same as those of $y$ in $\G_y$. 
\end{proof}

\begin{lemma}\label{basic2}
$f_\#:C_n(X)\to C_n(Y)$ is surjective.
\end{lemma}
\begin{proof}
Let $\G'$ be a $k$-simplex in $Y$, and $y\in\mathring\G'$.
Pick $x\in X$ such that $f(x)=y$. By Lemma \ref{basic1}, $f:\G_x\to\G'$ is an isomorphism.  It follows that for each $\Delta'$ such that $\G'=|\Delta'|$, 
there is an ordered $k$-simplex $\Delta$ in $X$ such that $f_\#(\Delta)=\Delta',$
so that $f_\#$ is surjective.
\end{proof}

\begin{remark}
The statement of Lemma \ref{basic2} is false in singular homology when $n>0$: for example, the figure $8$ is the image of the circle $S^1$ under a generically 1-to-1 immersion with one double point. 
A 1-simplex in the figure $8$ which turns a corner at the vertex of the $8$ is not the image of any 1-simplex in $X$.
\end{remark}
 
\subsection*{Triangulating $W^k(f)$}
 
Recall first that since $X\subset\RR^N$, we have $W_k(f)\subset (\RR^{N})^k$.
Let $y\in Y$. Suppose $\G_y$ is an $n$-simplex. Let $\x\in W^k(f)$ such that $f^{(k)}(\x)=y$, i.~e., $f(\pi_\ell(\x))=y$ for $\ell=1,\dots,k$.
Picking an order of the vertices of $\G_y$, say $\G_y=|[w_0,\ld,w_n]|$, determines an order for the vertices of each simplex $\G_{\pi_\ell(\x)}$
\begin{equation*}
\G_{\pi_\ell(\x)}=\l|\l[v^{(\ell)}_0,\ld,v^{(\ell)}_n\r]\r|,
\end{equation*}
where $f(v^{(\ell)}_j)=w_j$ for $j=0,\ld,n$.  Then the simplex
\begin{equation}\label{biggest}
\G^{(k)}_{\x}:=\l|\l[\bigl(v^{(1)}_0,\ld,v^{(k)}_0\bigr), \ld, \bigl(v^{(1)}_n,\ld,v^{(k)}_n\bigr)\r]\r|
\end{equation}
is contained in $W^k(f)$ since all its vertices are. We also denote the geometric simplex \eqref{biggest} by
\begin{equation*}
\left(\G_{\pi_1(\x)}\times\cd\times \G_{\pi_k(\x)}\right)/Y.
\end{equation*}

We claim that $\x$ is in the interior of $\G^{(k)}_{\x}$. 
Write $y$ in barycentric coordinates $y=\sum_j t_jw_j$, since $y$ is in the interior of $\G_y$, $t_j>0$ for $j=0,\dots,n$. By Lemma~\ref{cp} the barycentric coordinates of each point $\pi_\ell(\x)$
in $\G_{\pi_\ell(\x)}$ with the vertices ordered as described, are the same as those of $y$ in $\G_y$. Thus
\begin{equation*}
\x=\sum_j t_j (v_j^{(1)},\dots,v_j^{(k)})
\end{equation*}
lies in the interior of $\G^{(k)}_{\x}$ since $t_j>0$ for $j=0,\dots,n$.

\begin{lemma}\label{lem:unique.geom.simp}
Let $\x\in W^k(f)$. The geometric simplex $\G^{(k)}_{\mathbf{x}}$ given in \eqref{biggest} is the unique geometric simplex of $W^k(f)$ containing $\mathbf{x}$ in its interior,
\end{lemma}

\begin{proof}
Suppose $\mathbf{x'}$ lies in the interior of the $n$-simplex $\G^{(k)}_{\mathbf x}$. Then, with the notation of the construction,
\begin{equation*}
\mathbf x'= \sum_jt_j'\left(v^{(1)}_{j},\ld,  v^{(k)}_{j}\right),
\end{equation*}
for some $t_j'$, with $t_j'>0$ for $j=0,\ld,n$ and $\sum_jt_j'=1.$
Then $\pi_\ell(\x')=\sum_jt_j'v_{j}^{(\ell)}, $ and $\pi_\ell(\mathbf{x'})\in\mathring\G_{\pi_\ell(\x)}$,
so that $\G_{\pi_\ell(\mathbf{x'})}=\G_{\pi_\ell(\mathbf{x})}$. It follows that $\G^{(k)}_{\mathbf{x'}}=\G^{(k)}_{\mathbf{x}}$. 
Thus each point of $W^k(f)$ lies in the interior of a unique geometric simplex
\end{proof}

By Lemma~\ref{lem:unique.geom.simp} the notation $\G^{(k)}_{\mathbf x}$ is consistent with the notation $\G_x$ for the unique geometric simplex of $X$ containing $x$ in its interior.
Note that the ``$k$" in $\G^{(k)}_{\mathbf x}$ refers to the fact that we are speaking of a simplex in $D^k(f)$; it is not the dimension of the simplex.

\begin{proposition}\label{prop:sc.W}
The forgoing construction gives $W^k(f)$ the structure of a simplicial complex.
\end{proposition}

\begin{proof}
Omitting a vertex of a simplex $\G^{(k)}$ in $W^k(f)$ we obtain again a geometric simplex in $W^k(f)$, and by Lemma~\ref{lem:unique.geom.simp} 
for each point $\mathbf{x}\in W^k(f)$, $\G^{(k)}_{\x}$ is the unique geometric simplex containing $\x$ in its interior.
\end{proof}

We refer to the triangulation of $W^k(f)$ just described as the {\it triangulation associated to the simplicial map $f$}.

\subsection*{Triangulating $D^k(f)$}
Let $y\in Y$, suppose that $\G_y$ is an $n$-simplex, and suppose $\{x_1,\ld, x_k\}\subset f^{-1}(y)$ with $x_i\neq x_{i'}$ if $i\neq i'$. 
Thus, the point $\x=(x_1,\ld,x_k)$ is a point in $D^k_S(f)$.
By Lemma \ref{basic1},  the preimage of $\G_y$ contains the geometric
$n$-simplices $\G_{x_i}$, each of which is mapped isomorphically to $\G_y$.  Suppose that 
$\G_y=|[w_0,\ld,w_n]|$, and that $y=\sum_jt_jw_j$. Let $v^{(i)}_j$ be the unique vertex of $\G_{x_i}$ such that
$f(v^{(i)}_j)=w_j$, for $i=1,\ld,k$ and $j=0,\ld,n$. Note that it may happen that $v_j^{(i)}=v_j^{(i')}$ for some $i\neq i'$. 
By Lemma \ref{cp}, the barycentric coordinates of each point $x_i$ in $\G_{x_i}$, with the vertices ordered as described, are the same as those of $y$ in $\G_y$. Thus, 
the geometric $n$-simplex
\begin{equation}\label{1big}
\left|\left[\lb v^{(1)}_0,\ld, v^{(k)}_0\rb,\ld, \lb v^{(1)}_n,\ld, v^{(k)}_n\rb\right]\right|
\end{equation}
is contained in $D^k(f)$, and
\begin{equation*}
\x=(x_1,\ld, x_r)=\sum_jt_j\lb v^{(1)}_j,\ld, v^{(k)}_j\rb
\end{equation*}
lies in its interior: all of the $t_{j_i}$ are strictly positive since $y\in\mathring\G_{y}$.

As before, we denote the geometric simplex \eqref{1big} by $\G^{(k)}_{\x}$ or by
\begin{equation*}
\lb \G_{x_1}\times\cd\times\G_{x_k}\rb/Y.
\end{equation*}
Its interior is contained in $D^k_S(f)$, but some of its faces may not be. When $k=2$, this is more commonly written as $\G_{x_1}\times_Y\G_{x_2}$. 

Let us consider now more general points $\x\in D^k(f)$, where there may be repetitions in the $k$ components. Such points appear in the closure of $D^k_S(f)$. 
We claim that the simplex \eqref{biggest} constructed for $W^k(f)$ before
is contained in $D^k(f)$, and contains $\x$ in its interior. 

To see this, observe that by definition of $D^k(f)$ as the closure of $D^k_S(f)$, there is a sequence  
$\l(x^{(1)}_{n},\ld,x^{(k)}_n\r)_{n\in\NN}$ in $D^{k}_S(f)$ converging to $\mathbf{x}$. By local finiteness of the triangulation of $X$, 
by passing to a subsequence we may suppose each individual sequence $(x^{(\ell)}_{n})_{n\in\NN}$ is contained in the interior of a single simplex, which we denote by $\G^{(\ell)}$. 
Note that $\G^{(\ell)}\neq \G^{(\ell')}$ if $\ell\neq \ell'$ since $f$ is one-to-one on each simplex. Then by the argument of the first part of the construction, 
$$\l(\G^{(1)}\times\cd\times \G^{(k)}\r)/Y$$
is contained in $D^k(f)$. Since $\G_{\pi_\ell(\x)}\leq \G^{(\ell)}$ for $\ell=1,\ld,k$, 
it follows that the simplex \eqref{biggest} is contained in $D^k(f)$. Again, $\x$ is contained
in its interior, since the barycentric coordinates of each point $\pi_\ell(\x)$ in $|[v^{(\ell)}_0,\ld,v^{(\ell)}_n]|$ are equal to those of $y$ in $[w_0,\ld, w_n]$, which are all strictly positive.
 
By Lemma~\ref{lem:unique.geom.simp} $\G^{(k)}_{\mathbf{x}}$ is the unique geometric simplex of $D^k(f)$ containing $\mathbf{x}$ in its interior.

\paragraph{\textbf{Notation}} As described in the construction, each of the simplices $\D_\ell:=[v^{(\ell)}_0,\ld,v_n^{(\ell)}]$ appearing in \eqref{biggest} satisfies
$f_\#(\D_\ell)=[w_0,\ld,w_n]$. It will be useful to refer to the ordered $n$ simplex of $D^k(f)$ appearing in 
\eqref{biggest} as $(\D_1\times\cd\times\D_k)/Y.$

\begin{corollary}\label{simplcomp} The forgoing construction gives $D^k(f)$ the structure of a simplicial complex. 
In fact, it is a simplicial subcomplex of $W^k(f)$ with the triangulation associated to the simplicial map $f$.
\end{corollary}

\begin{proof} 
It is obvious that the simplices constructed in $D^k(f)$ are also simplices of $W^k(f)$. To prove that $D^k(f)$ is a simplicial subcomplex of $W^k(f)$,
it is necessary to show that if we omit any vertex of a simplex of $D^k(f)$ then the resulting simplex is still a simplex of $D^k(f)$.
This is also obvious, since by omitting a vertex of a simplex $\G^{(k)}$ we obtain a geometric simplex in the topological closure of $\G^{(k)}$.
\end{proof}

We also refer to this triangulation of $D^k(f)$ as the {\it triangulation associated to the simplicial map $f$}.    

\begin{proposition}\label{prop:eps.triang}
Let $Z^k$ be $W^k(f)$ or $D^k(f)$.
For each $k$ and $1\leq i\leq k$ the map $\ve^{i,k}\colon Z^k\to Z^{k-1}$ is a simplicial map, with respect to the triangulations associated to the simplicial map $f$.
\end{proposition}

\begin{proof}
It is clear from \eqref{biggest} that $\ve^{i,k}(\G^{(k)}_{\mathbf{x}})=\G^{(k-1)}_{\ve^{i,k}(\mathbf{x})}$. 
The two geometric simplices have the same dimension, and with respect to the barycentric coordinates defined above,  $\ve^{i,k}$ is the identity map. 
\end{proof}

\begin{remark}\label{sumcom}
Let $Z^k$ be $W^k(f)$ or $D^k(f)$.
Consider the map $f^{(k)}\colon Z^k\to Y$ defined by 
$f^{(k)}(\mathbf{x})=f\bigl(\pi_\ell(\mathbf{x})\bigr)$ for some $\ell$ with $1\leq\ell\leq k$. 
Evidently this is independent of the choice of $\ell$, so for convenience we take $\ell=1$. We have
\begin{equation*}
f^{(k)}\left(\sum_j t_j\Bigl(v_{j}^{(1)},v_{j}^{(2)},\ld,v_{j}^{(k)}\Bigr)\right)=f\Bigl(\sum_jt_jv^{(1)}_{j}\Bigr)=\sum_jt_jf\bigl(v^{(1)}_{j}\bigr),
\end{equation*}
the last equality because $f$ is linear on each simplex of $X$. Thus $f^{(k)}$ defines an isomorphism 
$\G^{(k)}_{\mathbf{x} }\to \G_{f(\pi_\ell(\mathbf{x}))}.$ Conversely, for any $n$-simplex $\G$ in $Y$, each component of any point $\mathbf{x}$ in $Z^k$ such 
that $f^{(k)}(\G^{(k)}_{\mathbf{x}})=\G_y$, lies in the interior of a simplex which is mapped isomorphically to $\G_y$.  
\end{remark}

Again, let $Z^k$ denote $W^k(f)$ or $D^k(f)$. 
Let $\D=[w_0,\ld,w_n]$ be an ordered $n$-simplex of $Y$. Denote by $C_n(Z^k)|_{\D}$ and $C^\alt_n(Z^k)|_{\D}$ the 
collection of $n$-chains, and alternating $n$-chains, respectively,  consisting of linear combinations of ordered $n$-simplices $\D^{(k)}$ such that 
$f^{(k)}_\#(\D^{(k)})=\D$. Every such simplex is of the form $(\D_1\times\cd\times\D_k)/Y$, where each 
$\D_\ell$ is a simplex of $X$ such that $f_\#(\D_\ell)=\D$. Note that by $Z^1$ we mean simply $X$, so that both $C_n(Z^1)|_\D$ and $C_n^\alt(Z^1)|_\D$ consist of 
the free abelian group generated by the $n$-simplices $\Delta'$ of $X$ such that $f_\#(\D')=\D$. We have
\begin{equation*}
\ve^{i,k}_\#\bigl((\D_1\times\cd\times\D_k)/Y\bigr)=(\D_1\times\cd\times\widehat{\D_i}\times\cd\times\D_{k})/Y,
\end{equation*}
which implies the following  simple but important lemma. 
\begin{lemma}\label{subcomp} 
$C_n(Z^\bu)|_{\D}$ and $C^\alt_n(Z^\bu)|_{\D}$ are subcomplexes of the complexes $C_n(Z^\bu)$ and $C^\alt_n(Z^\bu)$.\eop
\end{lemma}

\subsection{The sequence $\lb C_n(W^\bu(f)), \varrho^\bu\rb$ is a resolution of $C_n(Y)$.} 

Now we shall prove that the sequence \eqref{eq:wcc} is exact.

\begin{lemma}\label{lem:rho.rho}
$\rho^{k-1}_n\circ\rho^k_n=0$.
\end{lemma}

\begin{proof}
This is the usual computation. Denoting by $\widehat{v_n^{(i)}}$ the entry that has to be removed we have\setlength{\multlinegap}{0pt}
\begin{multline*}
\rho^{k-1}_n\circ\rho^k_n([(v_0^{(1)},\dots,v_0^{(k)}),\dots,(v_n^{(1)},\dots,v_n^{(k)})])\\
=\rho^{k-1}_n\Bigl(\sum_{i=1}^k(-1)^{i-1}[(v_0^{(1)},\dots,\widehat{v_0^{(i)}},\dots,v_0^{(k)}),\dots,(v_n^{(1)},\dots,\widehat{v_n^{(i)}},\dots,v_n^{(k)})]\Bigr)\\
=\sum_{j<i}(-1)^{i+j-2}[(v_0^{(1)},\dots,\widehat{v_0^{(j)}},\dots,\widehat{v_0^{(i)}},\dots,v_0^{(k)}),\dots,(v_n^{(1)},\dots,\widehat{v_n^{(j)}},\dots,\widehat{v_n^{(i)}},\dots,v_n^{(k)})]\\
+\sum_{j>i}(-1)^{i+j-3}[(v_0^{(1)},\dots,\widehat{v_0^{(i)}},\dots,\widehat{v_0^{(j)}},\dots,v_0^{(k)}),\dots,(v_n^{(1)},\dots,\widehat{v_n^{(i)}},\dots,\widehat{v_n^{(j)}},\dots,v_n^{(k)})].
\end{multline*}
The latter two summations cancel since after switching $i$ and $j$ in the second sum, it becomes the negative of the first.
\end{proof}

\begin{remark}
It is immediate from Lemma~\ref{lem:rho.rho} and the definition of $\varrho^k_n$ given in \eqref{eq:sign-trick} that we also have $\varrho^{k-1}_n\circ\varrho^k_n=0$.
Hence all the columns of the double complex \eqref{eq:dc.w} are chain complexes.
\end{remark}

\begin{lemma}\label{lem:f.rho}
$f_{\#,n}\circ\rho^2_n=0$.
\end{lemma}

\begin{proof}\setlength{\multlinegap}{0pt}
\begin{multline*}
f_{\#,n}\circ\rho^2_n\bigl([(v_0^{(1)},v_0^{(2)}),\dots,(v_n^{(1)},v_n^{(2)})]\bigr)=f_{\#,n}\bigl([(v_0^{(2)}),\dots,(v_n^{(2)})]-[(v_0^{(1)}),\dots,(v_n^{(1)})]\bigr)\\
=[(f(v_0^{(2)})),\dots,(f(v_n^{(2)}))]-[(f(v_0^{(1)})),\dots,(f(v_n^{(1)}))]=0
\end{multline*}
since $f(v_i^{(1)})=f(v_i^{(2)})$ for $i=1,\dots,n$.
\end{proof}

By Lemma~\ref{lem:rho.rho} and Lemma~\ref{lem:f.rho} the sequence \eqref{eq:wcc} is a chain complex. It is also a chain complex if we replace $\rho^k_n$ by $\varrho^k_n$.

\begin{lemma}\label{lem:W.splitting}
For each $n$, the complex  $(C_n(W^{\bu}(f)),\varrho^\bu_n)$ splits as a direct sum 
of subcomplexes:
\begin{equation}\label{split.w}
C_n(W^{\bu}(f))=\bigoplus C_n(W^{\bu}(f))|_{\D}.
\end{equation}
Here the direct sum is over ordered $n$-simplices $\D$ in $Y$. 
\end{lemma}
\begin{proof}
We have already seen, in Lemma \ref{subcomp}, 
that each of the $(C_n(W^{\bu}(f))|_{\D},\varrho^\bu_n)$ {\it is} a subcomplex. From the definition of $C_n(W^{\bu}(f))|_{\D}$ we have that
if $\Delta\neq\D'$, then
\begin{equation*}
C_n(W^{k}(f))|_{\D}\,\cap\, C_n(W^{k}(f))|_{\D'}=0.
\end{equation*}
\end{proof} 

\begin{proposition}\label{prop:wcc.exact}
For each $n\geq0$, the complex
\begin{equation*}
\dots\to C_n(W^k(f))\xrightarrow{\varrho^k_n}\cdots
\xrightarrow{\varrho^3_n}C_n(W^2(f))\xrightarrow{\varrho^2_n}C_n(X)\xrightarrow{f_{\#,n}}C_n(Y)\to0
\end{equation*}
is an acyclic resolution of $C_n(Y)$.
\end{proposition}

\begin{proof}
We have $C_n(Y)=\bigoplus \ZZ\cdot\D$, where the direct sum is over ordered $n$-simplices of $Y$. We claim that for each ordered $n$-simplex $\D$ of $Y$, the complex
\begin{equation*}
\dots\to C_n(W^k(f))|_{\D}\xrightarrow{\varrho^k_n}\cdots
\xrightarrow{\varrho^3_n}C_n(W^2(f))|_{\D}\xrightarrow{\varrho^2_n}C_n(X)|_{\D}\xrightarrow{f_{\#,n}}\ZZ\cdot\D\to0
\end{equation*}
is exact. Because of the splitting of $(C_n(W^\bu(f)),\varrho^\bu_n)$ as a direct sum of subcomplexes  $(C_n(W^\bu(f))|_{\D},\varrho^\bu_n)$ given by Lemma~\ref{lem:W.splitting}
this will prove the proposition.

Suppose $\D=[w_0,\ld,w_n]$. Let $\D'=[v_0,\ld,v_n]$ be an ordered $n$-simplex in $X$ such that $f_\#(\D')=\D$ with $f(v_j)=w_j$ for $j=1,\ld,n$. Let
\begin{equation*}
\D^{(k)}=\Bigl[(v_0^{(1)},\dots,v_0^{(k)}),\dots,(v_n^{(1)},\dots,v_n^{(k)})\Bigr]
\end{equation*}
be an ordered $n$-simplex in $C_n(W^k(f))|_{\D}$, that is, an ordered $n$-simplex in $W^k(f)$ such that $f^{(k)}(\D^{(k)})=\D$ and $f(v_j^{(i)})=w_j$ for $i=1,\ld,k$ and $j=1,\ld,n$.
Define the homomorphisms
\begin{gather*}
s_k^{\D'}\colon C_n(W^k(f))|_{\D}\to C_n(W^{k-1}(f))|_{\D}\\
s_k^{\D'}([(v_0^{(1)},\dots,v_0^{(k)}),\dots,(v_n^{(1)},\dots,v_n^{(k)})])=[(v_0,v_0^{(1)},\dots,v_0^{(k)}),\dots,(v_n,v_n^{(1)},\dots,v_n^{(k)})]
\end{gather*}
for $k\geq1$ and
\begin{align*}
s_0^{\D'}\colon \ZZ\cdot\D&\to C_n(X)|_{\D}\\
\D=[w_0,\ld,w_n]&\mapsto\D'=[v_0,\ld,v_n].
\end{align*}
It is straightforward to see that
\begin{itemize}
 \item $\ve^{1,k}_{\#,n}\circ s_k^{\D'}=Id$,
\item $\ve^{i+1,k+1}_{\#,n}\circ s_k^{\D'}=s_{k-1}^{\D'}\circ\ve^{i,k}_{\#,n}$,
\item $f_\#\circ s_0^{\D'}=Id$ 
\end{itemize}
Therefore $s_k^{\D'}$ with $k\geq0$ defines a contracting homotopy and the complex is acyclic.
\end{proof}

\subsection{The sequence $\lb C^\alt_n(D^\bu(f)), \epsilon^\bu\rb$ is a resolution of $C_n(Y)$.} 

Now it is the turn to prove that the sequence \eqref{eq:rcy} is exact.

\begin{lemma}\label{D2XY}$\ve^2_{\#,n}\left(C^\alt_n(D^2(f))\right)=\ker\left(f_{\#,n}\colon C_n(X)\to C_n(Y)\right)$.
\end{lemma}
\begin{proof} We have already seen in Lemma \ref{isacomplex} that 
\begin{equation*}
\ve^2_{\#,n}\left(C^\alt_n(D^2(f))\right)\subseteq\ker\left(f_{\#,n}:C_n(X)\to C_n(Y)\right).
\end{equation*}
Now we prove the opposite inclusion.

Suppose that $c:=\sum_im_i\Delta_i$ is a chain in $\ker f_{\#,n}\colon C_n(X)\to C_n(Y)$:
\begin{equation*}
\sum_im_if_{\#,n}(\Delta_i)=0.
\end{equation*}
Write $\|c\|:=\sum_i|m_i|$. We can assume that for each $i\neq j$ with $m_i\neq 0\neq m_j$, we have $\Delta_i\neq\Delta_j$. 
For each simplex
$\Delta_{i_1}$ with $m_{i_1}\neq 0$, let $\Delta_{i_2},\ld, \Delta_{i_r}$ be the simplices with non-zero coefficients in $c$ such that $f_{\#,n}(\Delta_{i_j})= f_{\#,n}(\Delta_{i_1})$.  
We can assume that $m_{i_1}>0$. 
Thus  $\sum_{j=1}^rm_{i_j}=0$. Without loss of generality we can assume that $m_{i_2}<0$. Let 
$\Delta_{i_1}=[v_0,\ld,v_n]$ and $\Delta_{i_2}=[v'_0,\ld,v'_n]$; our assumptions means that $f(v_j)=f(v'_j)$ for $j=0,\ld,n$. 
Thus, as explained in the construction of the triangulation of the $D^k(f)$ preceding Corollary~\ref{simplcomp},
$\D_{i_1}\times_Y\D_{i_2}=[(v_0,v'_0),\ld,(v_n,v'_n)]$
is an $n$-simplex in the triangulation of $D^2(f)$. Then 
\begin{equation*}
c':=\D_{i_1}\times_Y\D_{i_2}-\D_{i_2}\times_Y\D_{i_1}=[(v_0,v'_0),\ld,(v_n,v'_n)]-[(v'_0,v_0),\ld,(v'_n,v_n)]
\end{equation*}
is an alternating $n$-chain on $D^2(f)$ and
\begin{equation*}
\ve^2_{\#,n}(c')=\Delta_{i_1}-\Delta_{i_2}.
\end{equation*}
The chain $c-\ve^2_\#(c')$ lies in $\ker\,f_{\#,n}$, and moreover $\|c-\ve^2_{\#,n}(c')\|<\|c\|$. By induction on $\|c\|$,
we conclude that
\begin{equation*}
\ker\,f_{\#,n}\subset \ve^2_{\#,n}\left(C_n^\alt(D^2(f)\right).
\end{equation*}
Since the opposite inclusion always holds, this is an equality.
\end{proof} 

For any simplex $\D^{(k)}$ in the triangulation of $D^k(f)$, let 
\beq\label{defalt}\alt_\ZZ(\D^{(k)})=\sum_{\sigma\in S_k}\sign(\sigma)\sigma_\#(\D^{(k)}).\eeq
Clearly $\alt_\ZZ(\D^{(k)})$ is an alternating chain.  We use the subindex $\ZZ$ to distinguish this operator from
the more familiar projection operator $\alt_\QQ=\frac{1}{k!}\sum_{\sigma\in S_k}\sign(\sigma)\sigma_\#.$
\begin{lemma}\label{dirs} 
\begin{enumerate}[1.]
 \item If the simplex $\D^{(k)}$ in $D^k(f)$ is invariant under some non-trivial permutation $\sigma\in S_k$,  then $\D^{(k)}$ does not appear in any irredundant expression of any alternating chain $c$, 
and moreover $\alt_\ZZ(\D^{(k)})=0$.
\item Let $\D_y$ be an $n$-simplex in $Y$, and let $\D_{1},\ld,\D_{N}$ be simplices in $X$ for which
$f_\#(\D_i)=\D_y$, with $\D_i\neq \D_j$ for $i\neq j$.  For $1\leq i_1<\cd <i_k\leq N$, denote the sequence $i_1,\ld,i_k$ by $I$, and let 
$\D^{(k)}_I$ be the $n$-simplex $(\D_{i_1}\times\cd\times \D_{i_k})/Y$ in $D^k(f)$. 
Then $C^\alt_n(D^k(f))$ is freely generated over $\ZZ$ by the alternating chains $\alt_\ZZ(\D^{(k)}_{I})$ for such $I$.
\end{enumerate}
\end{lemma}

\begin{proof}
1. If $\D^{(k)}$ is invariant under some non-trivial permutation $\sigma\in S_k$, pick $i\neq j$
such that $\sigma(i)=j$. The fact that $\D^{(k)}$ is invariant under $\sigma$ means that $|\Delta^{(k)}|$ is contained
in the set $\{x_i=x_j\}\subset X^k$, and hence $\Delta^{(k)}$ is invariant under $(i,j)$. Suppose that the coefficient of $\D^{(k)}$ in some alternating chain $c$ on $D^k(f)$ is $m$. 
Because $(i,j)$ leaves $\D^{(k)}$ fixed, its coefficient in $(i,j)_\#(c)$ is also $m$. But because $c$ is alternating and $\sign(i,j)=-1$, we must have $m=-m$. Thus $m=0$.

If $\D^{(k)}$ is invariant under some non-trivial permutation, so is $\sigma_\#(\D^{(k)})$ for each $\sigma\in S_k$.  It follows that $\alt(\D^{(k)}_i)=0$.   
\vsn 2.  By part 1 of the lemma, every simplex $\D^{(k)}$ appearing non-trivially in an alternating chain is in the $S_k$ orbit of some
simplex $\D^{(k)}_{I}$ with $I$ as described. The 
argument just given shows that for any such simplex $\D^{(k)}_I$, the simplices $\sigma_\#(\D^{(k)}_I), \sigma\in S_k,$ are pairwise distinct, 
and thus that the coefficient of $\D^{(k)}_I$ in $\alt_\ZZ(\D^{(k)}_I)$ is equal to 1. Now let $c$ be an alternating simplicial $n$-chain on $D^k(f)$. 
We claim that $c$ is a linear combination of chains $\alt(\D^{(k)}_{I})$. To see this, suppose that $m\neq 0$ is the coefficient of $\D^{(k)}_{I}$ in an
irredundant expression of $c$ as linear combination of ordered simplices. Let $\|c\|$ be the sum of the absolute values of the coefficients of the simplices in such an expression. 
Because $c$ is alternating, for each $\sigma\in S_k$, the coefficient of $\sigma_\#(\D^{(k)}_{I})$ is $\sign(\sigma)m$. It follows that
\begin{equation*}
\|c-m\alt_\ZZ(\D^{(k)}_\x)\|<\|c\|.
\end{equation*}
Hence, by induction on $\|c\|$, $c$ is a linear combination of alternating chains $\alt_\ZZ(\D^{(k)}_I)$.  

That there are no relations among distinct chains $\alt_\ZZ\D{(k)}_I$ is obvious: if $I=\{i_1,\ld,i_k\}$ and $I'\neq I$
then the simplex $(\D_{i_1}\times\cd\times \D_{i_k})/Y$ does not appear in $\alt_\ZZ\D^{(k)}_{I'}$. 
\end{proof} 

\begin{remark}
The lemma shows that $C^\alt_n(D^k(f))$ is equal to the image of the ``Alternation Operator" 
$\alt_\ZZ:C_n(D^k(f))\to C_n^\alt(D^k(f))$. 
We learned this from Kevin Houston's paper \cite{houstontop}.  
This is a special feature of  this $S_k$-action. In contrast, consider the action of $S_3$ on $\RR\times\RR^3$ defined by 
$\sigma(t,x_1,x_2,x_3)=((\text{sign }\sigma )t,x_{\sigma(1)},x_{\sigma(2)},x_{\sigma(3)})$. The point $(1,x,x,x)$
is then invariant only under even permutations. If we apply $\alt_\ZZ$ to $(1,x,x,x)$, we get 
$3((1,x,x,x)-(-1,x,x,x))$. The alternating $0$-chain $(1,x,x,x)-(-1,x,x,x)$ is not in the image of $\alt_\ZZ.$
\end{remark}

\begin{lemma} 
For each $n$, the complex  $(C_n^\alt(D^{\bu}(f)),\epsilon^\bu_n)$ splits as a direct sum 
of complexes:
\beq\label{splitnew}C_n^\alt(D^{\bu}(f))=\bigoplus C_n^\alt(D^{\bu}(f))|_{\D}.\eeq
Here the direct sum is over ordered $n$-simplices $\D$ in $Y$. 
\end{lemma}
\begin{proof}
We have already seen, in Lemma \ref{subcomp}, 
that each of the $(C_n^\alt(D^{\bu}(f))|_{\D},\epsilon^\bu_n)$ {\it is} a subcomplex. If $\Delta\neq\D'$, then
\begin{equation*}
C_n^\alt(D^{k}(f))|_{\D}\,\cap\, C_n^\alt(D^{k}(f))|_{\D'}=0.
\end{equation*}
The splitting \eqref{splitnew} now follows from Lemma \ref{dirs}. 
\end{proof} 
Now we prove the main theorem.  

\begin{proposition}\label{main} 
For each $n\geq 0$, the complex of alternating simplicial chains
\begin{equation*}
\dots\to C_n^\alt(D^k(f))\xrightarrow{\epsilon^{k}_n}\cdots
\xrightarrow{\epsilon^{3}_n}C_n^\alt(D^2(f))\xrightarrow{\epsilon^{2}_n}C_n(X)\xrightarrow{f_{\#,n}}C_n(Y)\to0,
\end{equation*}
is an acyclic resolution of $C_n(Y)$. 
\end{proposition}
\begin{proof} 
We have $C_n(Y)=\bigoplus \ZZ\cdot\D$, where the direct sum is over $n$-simplices of $Y$. We claim that for each $n$-simplex $\D$ of $Y$,
the complex
\begin{equation}\label{fincom1}
\dots\to C_n^\alt(D^k(f))|_\D\xrightarrow{\epsilon^{k}_n}\cdots
\xrightarrow{\epsilon^{3}_n}C_n^\alt(D^2(f))|_\D\xrightarrow{\epsilon^{2}_n}C_n(X)|_\D\xrightarrow{f_{\#,n}}\ZZ\cdot\D\to0
\end{equation}
is exact. 
Because of the splitting of $(C^\alt_n(D^{\bu}(f)), \epsilon^\bu)$ as a direct sum of subcomplexes
$(C_n^\alt(D^{\bu}(f))|_{\D},\e_n^\bu)$, this will prove the proposition.  Let $\D_1,\ld,\D_N$ be the distinct simplices $\D'$ of $X$ such that $f_\#(\D')=\D$. 
To prove exactness of \eqref{fincom1}, we show it is chain-isomorphic to a well-known exact complex, namely the augmented oriented simplicial 
chain complex of the standard $N-1$-simplex $\mathbf{\D^{N-1}}:=[e_1,\ld,e_N]$, with degree shifted by 1,  $C^{\text{or}}_\bu(\mathbf{\D^{N-1}})[1]$,
\begin{equation}\label{fincom2}
\dots\longrightarrow C^{\text{or}}_{k-1}(\mathbf{\D^{N-1}})\longrightarrow\cdots\longrightarrow C^{\text{or}}_1(\mathbf{\D^{N-1}})\longrightarrow C^{\text{or}}_0(\mathbf{\D^{N-1}})\longrightarrow\ZZ\longrightarrow0
\end{equation}
where $\ve(\sum_im_ie_i)=\sum_im_i$. In fact to obtain a {\it bona fide} chain isomorphism, we would have to modify the sign of the boundary operators in \eqref{fincom2}. 
Since our aim is simply to prove exactness of \eqref{fincom1}, it is enough instead to construct an isomorphism $\pp_k:C_n^\alt(D^k(f))|_\D\to C^{\text{or}}_{k-1}(\mathbf{\D^{N-1}})$ for each $k$ such that 
\beq\label{commute}\pp_{k-1}\circ\e^k_n=(-1)^{k+n-1}\p\circ\pp_k.\eeq
Recall from Lemma \ref{dirs} that $C^\alt_n(D^k(f))|_\D$ is freely generated by the $n$-chains $\alt_\ZZ(\D_{i_1}\times\cd\times \D_{i_k})/Y$, where $1\leq i_1<\cd<i_k\leq N$. 
When $k=1$, this of course just reduces to the single simplex $\D_{i_1}$. We map each such generator to the oriented $k-1$ face $[e_{i_1},\ld,e_{i_k}]$ of $\mathbf{\D^{N-1}}$, obtaining in this way an
isomorphism $$\pp_k:C^\alt_n(D^k(f))|_\D\to C^{\text{or}}_{k-1}(\mathbf{\D^{N-1}}).$$
And we define $\pp_0(m\D)=m$. To prove \eqref{commute}, 
recall first that $\e^k_n=(-1)^{n}\ve^k_{\#,n}.$ Since $\ve^k_{\#,n}$ maps alternating chains to alternating chains, so does $\e^k_n$; moreover since
\begin{equation*}
\ve^k_{\#,n}((\D_{i_1}\times\cd\times \D_{i_k})/Y)=(\D_{i_1}\times\cd\times \D_{i_{k-1}})/Y,
\end{equation*}
it follows that $\e^k_n(\alt_\ZZ(\D_{i_1}\times\cd\times \D_{i_k})/Y)$ must be a linear combination of chains
\begin{equation*}
\alt_\ZZ((\D_{i_1}\times\cd\times\widehat{\D_{i_\ell}}\times\cd\times \D_{i_k})/Y)
\end{equation*}
with $1\leq \ell\leq k$, and it is necessary only to determine the coefficient of each. The simplex $(\D_{i_1}\times\cd\times\widehat{\D_{i_\ell}}\times\cd\times \D_{i_k})/Y$ 
itself appears only once in the expression of $\ve^k_{\#,n}(\alt_\ZZ((\D_{i_1}\times\cd\times \D_{i_k})/Y)$ as a linear combination of simplices, namely as
\begin{equation*}
\ve^k_{\#,n}\sign(\sigma_{k,\ell})(\sigma_{k,\ell})_\#((\D_{i_1}\times\cd\times \D_{i_k})/Y),
\end{equation*}
where $\sigma_{k,\ell}$ is the permutation taking the ordered set $(0,\ld,k)$ to $(0, \ld,\hat \ell,\ld, k,\ell)$. 
This permutation has sign $(-1)^{k-\ell}$. It follows that this, multiplied by $(-1)^{n}$, is the coefficient of 
$\alt_\ZZ((\D_{i_1}\times\cd\times\widehat{\D_{i_\ell}}\times\cd\times \D_{i_k})/Y)$ in $\e^k(\alt_\ZZ(\D_{i_1}\times\cd\times \D_{i_k})/Y)$.  
The coefficient of $[e_{i_1},\ld,\widehat{e_{i_\ell}},\ld,e_{i_k}]$ in $\p\pp_k(\alt_\ZZ((\D_{i_1}\times\cd\times \D_{i_k})/Y)$ is $(-1)^{\ell-1}$, thus showing that \eqref{commute} holds,
and completing the proof.
\end{proof}

Putting together Proposition~\ref{prop:eps.triang}, Proposition~\ref{prop:wcc.exact} and Proposition~\ref{main} we obtain the proof of Theorem~\ref{thm:mt}, completing the construction of the GVZSS and the ICSS.

\subsection{Comparison with other proofs}
The first construction of the (cohomological) ICSS, for rational cohomology only, in \cite{gm}, was based on the exactness of the sequence of
locally constant sheaves
$$\xymatrix{0\ar[r]&\QQ_Y\ar[r]& f_*(\QQ_X)\ar[r]&f^{(2)}_*(\alt \QQ_{D^2(f)})\ar[r]&f^{(3)}_*(\alt \QQ_{D^3(f)})\ar[r]&\cd}.$$
Exactness was proved by the same argument as used here in Proposition \ref{main}. Goryunov in \cite{gor} 
used a geometric realisation $Y_G$ of the semi-simple object $D^\bu(f)$ due to Vassiliev and described in \cite{vassiliev}: 
it is formed by beginning with $X$ embedded in some high-dimensional Euclidean space, and
then adding, for each pair of points $x_1,x_2$ with $f(x_1)=f(x_2)$,  the 1-simplex $[x_1,x_2]$, for each triple $x_1, x_2, x_3$ with $f(x_1)=f(x_2)=f(x_3)$ the 2-simplex $[x_1,x_2,x_3]$, etc. 
The embedding of $X$ in $\RR^N$ is chosen so general that none of these added simplices meet one another except by the standard face inclusions. 
There is then a natural surjection  $Y_G\to Y$, and for each $y\in Y$ the preimage in $Y_G$ is a simplex. From this it follows that $Y_G\to Y$ is a homotopy equivalence, 
so that the homology of $Y$ may be computed as the homology of $Y_G$. Goryunov showed in \cite{gor} that the alternating  homology $AH_q(D^k(f))$ was isomorphic to the 
relative homology $H_{q+k-1}(Y^k_G, Y^{k-1}_G)$, where $Y^k_G$ is $X$ together with the added simplices of dimension $\leq k$, and the ICSS \eqref{eq:icss} is thus reduced to the spectral 
sequence for the homology of the filtered space $Y_G$. 

Our argument here is clearly closely related to Goryunov's proof. It seems likely that the total complex of the double complex $(C^\alt_\bu(D^\bu(f)),\p, \e^\bu)$ is isomorphic to the simplicial chain complex of some 
triangulation of $Y_G$. It would be interesting to prove this.

The construction of the GVZSS in \cite{GVZ:BNSASPS} for a closed map $f\colon X\to Y$ uses the fibred join $X*_YX$ to define the \textit{join space} $J^f(X)$ as the quotient
space of the disjoint union of spaces
\begin{equation*}
J^f_p(X)=\underbrace{X*_Y\cdots*_YX}_{\text{$p+1$ times}},\quad p=0,1,\dots
\end{equation*}
identifying $J^f_{p-1}(X)$ with each of its images $\phi_i(J^f_{p-1}(X))$ in $J^f_p(X)$ for $i=0,\ld,p$, where $\phi_i$ is the natural embedding which does not have the $i$'th copy of $X$ in its image.
Hence, there is a natural filtration of the join space $J^f(X)$ given by $\tilde{J}^f_0(X)\subset \tilde{J}^f_1(X)\subset\cdots\subset\tilde{J}^f_p(X)\subset\cdots\subset\tilde{J}^f(X)$
where $\tilde{J}^f_p(X)$ is the image of $J^f_p(X)$ in $J^f(X)$ under the quotient. The authors proved that the quotient space $\tilde{J}^f_p(X)/\tilde{J}^f_{p-1}(X)$ is homotopically equivalent to the
$p$'th suspension $S^p(W^{p+1}(f))$ of $W^{p+1}(f)$. There is a natural map $F\colon J^f(X)\to Y$ induced by $f$. 
It is proved that the fibres of $F$ are homologically trivial and by the Vietoris-Begle Theorem $H_*(J^f(X))\cong H_*(Y)$. The GVZSS \eqref{eq:gvzss} follows from the spectral 
sequence associated to the filtered space $J^f(X)$.
Since this construction uses the homotopic equivalence $\tilde{J}^f_p(X)/\tilde{J}^f_{p-1}(X)\simeq S^p(W^{p+1}(f))$ and the isomorphism 
$\tilde{H}_{p+q}(S^p(W^{p+1}))\cong \tilde{H}_q(W^{p+1})$ it is very difficult to give the differentials even of the first page of the spectral sequence. In contrast, our construction gives
such differentials explicitly.

\begin{remark}[Simplicial vs Singular]
Our theorem relates the alternating simplicial homology of the $D^k(f)$ to the simplicial homology
of the image, $Y$. It is well known that singular and simplicial homology coincide. In fact a variant of the standard proof shows that the same goes
for alternating simplicial homology and alternating singular homology.
\end{remark} 

\section{Cohomology}
The complex 
$\Hom_\ZZ(C^\alt_n(D^{\bu+1}(f),\ZZ)$
is acyclic, since $C^\alt_n(D^{\bu+1}(f))$ is a resolution of a free $\ZZ$-module, $C_n(Y)$. This shows that 
there is an ICSS also for the cohomology of the image, if we define alternating cohomology to be the homology of the complex
$\Hom_\ZZ(C^\alt_\bu(D^k(f)))$.  
\vsn 
An apparently different approach is to define alternating cochains by analogy with the definition of
alternating chains:
\begin{equation*}
C_\alt^n(D^k(f)):=\{\psi\in C^n(D^k(f)):\sigma^\#(\psi)=\sign(\sigma)\psi\ \text{ for all } \sigma\in S_k\}.
\end{equation*}
In fact, thanks to Lemma \ref{dirs},
the two approaches are equivalent. Given $\psi\in\Hom(C^\alt_n(D^k(f)),\ZZ)$, define a cochain 
$\alt_\ZZ^*\psi\in C^n(D^k(f))$ by 
composing with $\alt_\ZZ:C_n(D^k(f))\to C_n^\alt(D^k(f))$. If $\sigma\in S_k$ we check that for any chain 
$c\in C_n(D^k(f))$, 
\begin{align*}\alt_\ZZ^*\psi(\sigma_\#c)&=\psi(\alt_\ZZ(\sigma_\#(c))=\psi\left(\sum_{\tau\in S_k}\sign(\tau)\tau_\#\sigma_\#(c)\right)\\
&=\sign(\sigma)\psi\left(\sum_{\tau\in S_k}\sign(\tau)\sign(\sigma)(\tau\circ\sigma)_\#(c)\right)\\
&=\sign(\sigma)(\psi)(\alt_\ZZ(c))=\sign(\sigma)\alt_\ZZ^*\psi(c).
\end{align*}
Thus $\alt_\ZZ^*\psi\in C_\alt^n(D^k(f))$. In fact $\alt_\ZZ^*$ gives a homomorphism of cochain complexes
\begin{equation*}
\Hom_\ZZ(C^\alt_\bu(D^k(f)),\ZZ)\to C_\alt^\bu(D^k(f)),
\end{equation*}
and thus a homomorphism of cohomology groups. 

It has an inverse $\theta:C_\alt^\bu(D^k(f))\to\Hom_\ZZ(C^\alt_\bu(D^k(f)),\ZZ)$, defined as follows. Suppose that $\pp\in C^n_\alt(D^k(f))$ and $c\in C_n^\alt(D^k(f))$.
We can write $c=\sum_im_i\alt_\ZZ(c_i)$, where each $c_i$ is a simplex. This  representation is not unique, because for any $\sigma_i\in S_k$, the alternating chain $\alt_\ZZ(c_i)$ can also be written as 
$\sign(\sigma_i)\alt_\ZZ(\sigma_{i\#}(c_i))$. However,  because $\pp$ is alternating, we have
\begin{equation*}
\sum_i m_i\pp(c_i)=\sum_i\pp\bigl(m_i\sign(\sigma_i)\sigma_{i\#}(c_i)\bigr),
\end{equation*}
so that $\theta(\pp)$ is well-defined by the formula
\begin{equation*}
\theta(\pp)\Bigl(\sum_im_i\alt_\ZZ(c_i)\Bigr)=\sum_im_i\pp(c_i).
\end{equation*}
Once again, $\theta$ is a homomorphism of cochain complexes. 

Since, for any simplex $c$,
\begin{equation*}
\alt_\ZZ^*(\theta(\pp))(c)=\theta(\pp)(\alt_\ZZ(c))=\pp(c),
\end{equation*}
and for any alternating chain $c=\sum_im_i\alt_\ZZ(c_i)$,
\begin{equation*}
\theta({\alt_\ZZ^*(\psi)})(c)=\sum_im_i\alt_\ZZ^*(\psi)(c_i)=\sum_im_i\psi(\alt_\ZZ(c_i))=\psi(c),
\end{equation*}
$\theta$ and $\alt_\ZZ^*$ are mutually inverse. We conclude that
\begin{equation*}
H^*\left(\Hom_\ZZ(C^\alt_\bu(D^k(f)),\ZZ)\right)\simeq H^*(C^\bu_\alt(D^k(f)).
\end{equation*}

\section{Applications in Singularity Theory}\label{singularitytheory}
In this section we briefly survey some applications of the ICSS in singularity theory. First, 
finite analytic and subanalytic maps  can all be triangulated, thanks to the following theorem of Hardt in \cite{hardt}. 
\begin{theorem}{\em (\cite[Theorem 3]{hardt})} 
\begin{enumerate}[1.]
 \item Let $P\subset\RR^n$ be a closed finite-dimensional sub-analytic set and $N\subset \RR^m$ a real analytic
space, and let $f\colon P\to N$ be a proper light subanalytic map.  Then there exist simplicial complexes $\s G$ and $\s H$ and homeomorphisms $g\colon\s G\to P$ 
and $h\colon\s H\to f(P)$, and a simplicial map $p\colon\s G\to \s H$ such that $f=h\circ p\circ g^{-1}$.
\item Suppose, in the same situation, that $\s Q$ and $\s R$ are locally finite families of closed subanalytic sets in $P$ and $f(P)$, respectively. Then
the simplicial complexes $\s G$ and $\s H$, and the homeomorphisms $g$ and $h$, can all be chosen so that $g^{-1}(Q)$ is a subcomplex of 
$\s G$ and $h^{-1}(R)$ is a subcomplex of $\s H$ for all $Q\in \s Q$ and $R\in \s R$.
\end{enumerate}
\end{theorem}
A map is {\it light} if the preimage of a discrete set is discrete. Every real analytic space (and hence every complex analytic space) is subanalytic, and the same goes for analytic maps.    
Thus, our construction of the ICSS applies, for example, to a stable perturbation $f_t$ of an $\s A$-finite map-germs $(\FF^n,S)\to (\FF^p,0)$ when $n\leq p$, and, 
when $n\geq p$, to the restriction to the critical space of a stable perturbation of an $\s A$-finite germ. 
The ICSS calculates the homology of the image in the former case, and of the discriminant (set of critical values) in the latter. 
The ICSS was first developed in this context. 

If $f_0$ is a map-germ $f_0:(\CC^n,0)\to (\CC^p,0)$ with $n<p$,  of corank $\leq 1$ (i.e. $\dim \ker df_0=1$) then $D^k(f_0)$ is an isolated complete 
intersection singularity (ICIS) provided its expected dimension, $n-(k-1)(p-n)$ is non-negative, see \cite{marar-mond}. If $f_t$ is a stable perturbation of 
$f_0$ then each $D^k(f_t)$ is smooth, and of expected dimension. It is thus a Milnor fibre of an ICIS, and hence, crucially, has reduced homology only in middle dimension (i.e. $n-(k-1)(p-n)$),
see \cite{hamm} (or \cite{looijenga} for an account in English). In \cite[\S2]{gor}, Goryunov shows that the homology $AH_*(X)$ of the alternating chain complex on the 
Milnor fibre $X$ of an $S_k$-invariant ICIS coincides with the alternating part of its regular homology, $H_*^\alt(X)$, and thus it too is concentrated in middle dimension. 
From this it follows that the ICSS collapses at $E^1$, and the homology of the image $Y_t$ can be read off from it. To simplify the resulting formula, let $r:=p-n$. Then if, first, $r>1$, we have
\begin{equation}\label{yet}
H_q(Y_t)=\begin{cases}
          \ZZ&\text{if $q=0$,}\\
          H_{q-k+1}^\alt(D^k(f_t))&\text{if $q=n-(k-1)(r-1)\geq 0$,}\\
          0&\text{otherwise}
         \end{cases}
\end{equation}
If $r=1$, then all the multiple point spaces contribute to the $n$'th homology of the image, and the graded module  
arising from the resulting filtration on $H_n(Y_t)$ satisfies
\begin{equation}\label{yeti}
\text{Gr}H_n(Y_t)\simeq\bigoplus_{k=2}^{n+1}H^\alt_{q-r(k-1)}(D^k(f_t)).
\end{equation}
It is known, by a Morse-theoretic argument, that in this case the image has the homotopy type of a wedge of 
$n$-spheres.

In \cite[Theorem 4.6] {houstontop}, Kevin Houston proved the remarkable theorem that even without the hypothesis on the corank of $f$, $AH_*(D^k(f_t))$ is concentrated in middle dimension, 
so that the ICSS collapses at $E^1$ and the formulae \eqref{yet} and \eqref{yeti}, with $AH_*$ replacing $ H_*^\alt$,  continue to hold. This is all the more remarkable because when $f_0$ 
has corank $>1$, the multiple point spaces $D^k(f_0)$ are no longer complete intersections, and, indeed,
examples (see e.g. \cite{mond15}) show that their homology is not necessarily concentrated in middle dimension. It is only the alternating homology that
has this pleasant property.
Houston has to replace $ H_*^\alt(D^k(f_0))$ by $AH_*(D^k(f_0))$ since it is not known whether 
the two coincide when $D^k(f_0)$ is not an ICIS. Note, however, that for rational homology, the two always coincide. An example (not from singularity theory) where they do not
coincide is given by the quotient map $q:B^2\to\RR\PP^2$, where $B^2$ is the closed unit disc in $\RR^2$ and the map $q$ identifies antipodal points on the boundary $S^1$. Here $D^2(q)$ is isomorphic to $S^1$ itself, and its involution is identified with the antipodal map. Evidently 
 $ H_0^\alt(D^2(q))$ is trivial, but one calculates that $AH_0(D^2(q))\simeq\ZZ/2\ZZ$.   


\end{document}